\documentclass[12pt,leqno]{amsart}
\usepackage{amssymb,amsthm,amsmath,latexsym}
\usepackage{amsmath,amsthm,amssymb,amsfonts,amscd, amsbsy,mathtools}

\usepackage[user,titleref]{zref}

\usepackage{lineno,hyperref}
\modulolinenumbers[5]
\usepackage{microtype}
\everymath{\displaystyle}
\usepackage[all]{xy}
\usepackage{soul}
\usepackage[normalem]{ulem}
\newtheorem{thm}{Theorem}[section]
\makeatletter
\usepackage{enumitem}

\newtheorem{theorem}{Theorem}
\newenvironment{lettertheorem}{%
  \theorem
}{%
  \endtheorem
}
\makeatother

 \newtheorem{lem}[thm]{Lemma}
 \newtheorem{prop}[thm]{Proposition}
 \newtheorem*{defn}{Definition}
 \theoremstyle{remark}
 \newtheorem{rem}[thm]{\sc Remark}
 \newtheorem*{ex}{\sc Example}

\usepackage{color}

\title[On Pro-$p$ Cappitt Groups with finite exponent]
 {On Pro-$p$ Cappitt Groups with finite exponent}
\author[Porto]{Anderson Porto}
\author[Lima]{Igor Lima}
\address{Instituto de Ciência e Tecnologia - ICT \\ 
Universidade Federal dos Vales do Jequitinhonha e Mucuri \\ Diamantina - MG, 39100-000 \\ Brazil}
\address{ Departamento de Matem\'atica, Universidade de Bras\'ilia,
Brasilia-DF, 70910-900 Brazil }
\email{(Porto) ander.porto@ict.ufvjm.edu.br}
\email{(Igor) igor.matematico@gmail.com}



\thanks{The first author was partially supported by FAPDF, Brazil. The second author was partially supported by DPI/UnB, FEMAT Proc. 054/2022, FAPDF, Brazil.}

\dedicatory{To Pavel Zalesskii on his 60th birthday}

\begin{document}

\maketitle

\begin{abstract}
A pro-$p$ Cappitt group is a pro-$p$ group $G$ such that $\tilde{S}(G) = \overline{ \langle L \leqslant_c G \,\mid \,\, L \ntriangleleft G \rangle}$ is a proper closed subgroup (i.e. $\tilde{S}(G) \neq G$). In this paper we prove that non-abelian pro-$p$ Cappitt groups whose torsion subgroup is closed has finite exponent. This result is a natural continuation of main result of the first author\cite{Por19}. 
We also prove that in a pro-$p$ Cappitt group its subgroup commutator is a procyclic central subgroup. Finally we show that pro-$2$ Cappitt groups of exponent $4$ are pro-$2$ Dedekind groups. These results are pro-$p$ versions of the generalized Dedekind groups studied by Cappitt (see Theorem 1 and Lemma 7 in \cite{Cap71}). \end{abstract} 

\maketitle
\section{Introduction}

Profinite groups are totally disconnected compact Hausdorff topological groups. Such groups can be seen as projective limits of finite groups (see Ribes-Zalesskii \cite{RZ10} or Wilson \cite{wilson98}).

We recall that a totally disconnected compact Hausdorff group $G$ is called a profinite torsion group (or periodic) if all of its elements are of finite order.

One motivation for studying these groups is the following open question due Hewitt-Ross (see Open Question 4.8.5b in \cite{RZ10}): Is a profinite torsion group necessarily of finite exponent?

In celebrated paper \cite{Zelm92}, Zel'manov proves that a finitely generated pro-$p$ torsion group is finite. Therefore the question above is open for infinitely generated torsion groups.

In \cite{Herf80} and \cite{Herf82}, Herfort proves that a profinite group $G$ whose order is divisible by infinitely many different primes has a procyclic subgroup with the same property. Otherwise if $G$ is a profinite torsion group then the order of $G$ is divisible by only finitely many distinct primes.

As main result in this paper we prove that a infinitely generated non-abelian pro-$p$ Cappitt group has finite exponent. This result is the profinite version of second part of Theorem 1 in \cite{Cap71}. In \cite{Por19} the first author of this paper proves the first part of the our main result (see Theorem \ref{teoA}). Also we obtained other related results of independent interest (see Theorem \ref{teoB}).

\section{Preliminaries}

The standard notation is in accordance with \cite{Por19}, \cite{RZ10} and  \cite{wilson98}.

We recall that an abstract group (finite or infinite) is called a \textbf{Dedekind group} if every subgroup is normal. As in \cite{Cap71}, for any group $G$, we define $S(G)$ to be the subgroup of $G$ generated by all the subgroups which are not normal in $G$, i.e. $S(G) = <N | \ N \leq G \ \text{and } \ N \not\trianglelefteq G >$. Note that $G$ is a Dedekind group if and only if $S(G)=\{1\}$. Also $S(G)$ is a characteristic subgroup of $G$ and the quotient group of $G$ by $S(G)$ is a Dedekind group. A group is to be a  \textbf{Cappitt group} if it satisfies $S(G)\neq G$, these groups are generalizations of Dedekind groups (see for example Kappe and Reboli \cite{kappe} or Cappitt \cite{Cap71}). 

In the profinite version, as in \cite{Por19} and \cite{PB19}, a \textbf{profinite Dedekind group} will be a profinite group in which every closed subgroup is normal. A profinite non-abelian Dedekind group will be called a \textbf{profinite Hamiltonian group}.

\begin{rem} \label{remark prociclica}
If each procyclic subgroup of a profinite group $G$ is normal, then G is a profinite Dedekind group (see Remark 1, p. 90-91 in \cite{PB19}).
\end{rem} 

A. Porto and V. Bessa \cite{PB19} classified these groups as follow

\begin{thm} A profinite group $G$ is Dedekind if, and only if, $G$ is  abelian or there exists a finite set of odd primes $J$ and a natural number $e$ such that $$G \cong Q_8 \times \tilde{E} \times \displaystyle\prod_{p \in J} \left( \displaystyle\prod_{i=1}^{e} \left( \displaystyle\prod_{m(i,p)}\,C_{p^i}\right)\right),$$ where $Q_8$ is the quaternion group of order $8,$ $\tilde{E}$ is an elementary abelian pro-$2$ group and each $m(i,p)$ is a cardinal number. In particular, if $G$ is a profinite Hamiltonian group then $G$ has finite exponent.
\label{dedekindBP}
\end{thm}

Define $G\,'$ as the topological closure in $G$ of the abstract commutator subgroup $[G, G] = \left\langle [g,h]| g,h \in G\right\rangle$, where $[g,h]=g^{-1}h^{-1}gh$ is the commutator of the elements $g, h \in G$. The abelianization of group $G$ will be denoted by $G^{ab} = G/G'=G/\overline{[G,G]}$. If $G$ is a profinite group, $G^{ab}$ is an abelian profinite group.

\begin{defn}
    A \textbf{profinite Cappitt group} is a profinite group $G$ such that $\tilde{S}(G) = \overline{ \langle L \leqslant_c G \,|\,\, L \ntriangleleft G \rangle}$ is a proper subgroup of $G$.
\end{defn}

This profinite version was defined firstly in \cite{Por19}. For finite groups we have that $\tilde{S}(G)=S(G).$

Using Theorem \ref{dedekindBP}, Porto gave the following characterization of profinite Cappitt groups.

\begin{thm} Let $G$ be a profinite group satisfying $\{1\} \neq \tilde{S}(G) \neq G.$ Then $G$ can be expressed as direct product $\tilde{H} \times \tilde{K}$, where $\tilde{H}$ is a pro-$p$ Cappitt group for some prime $p$ and $\tilde{K}$ is a profinite Dedekind group that does not contain elements of order $p$. Such $G$ is nilpotent profinite of class at most $2$.
\label{mainPorto}
\end{thm}

Denote by $\mathbb{Z}_{p}$ the $p$-adic integers. Consider $|z|$ the order of an element $z$ of a group $G$. For a group $G$ we denote $tor(G) = \{x \in G\,|\, |x| < \infty \}$ the subset of the elements of torsion of $G$. As in the abstract case (see 16.2.7 in \cite{Kargapolov}), when $G$ is a nilpotent pro-$p$ group, then $tor(G)$ is a subgroup of $G$ (not necessarily closed). A main result of this paper is the following

\begin{lettertheorem} \label{teoA} Let $G$ be a non-abelian pro-$p$ Cappitt group. Then $G\,'$ is a procyclic central subgroup. Moreover, if $tor(G) \leqslant_{c} G$ then $G$ has finite exponent.
\end{lettertheorem}

\label{exampleperiodic}

\begin{rem} \label{matalogo}
The Theorem \ref{teoA} has immediate proof when $G$ is a pro-$2$ Hamiltonian group since $G\,' = C_2$ is central because $G$ has nilpotent class $2$ (see Corollary $1$ in \cite{Por19}), and also  $G$ has finite exponent by Theorem \ref{dedekindBP}. The condition to be non-abelian is necessary in the Theorem \ref{teoA} since $\mathbb{Z}_{p}$ is an abelian non-periodic group satisfying all hypothesis.
\end{rem}

Now we present some preliminaries results that will are used throughout in the proof of the main theorem.

\begin{lem} \label{limitesdeS} Let $G = \displaystyle\lim_{\longleftarrow\atop{i \in \mathcal{I}}}\,\{G_i, \phi_{ij}\}$ where $\{G_i, \phi_{ij}, \mathcal{I}\}$ is a surjective inverse system of finite groups (with discrete topology). Then $$\tilde{S}(G) = \displaystyle\lim_{\longleftarrow \atop{i \in \mathcal{I}}}\,\{S(G_i), \phi_{ij}\vert_{_{S(G_i)}} \},$$ where $\{S(G_i), \phi_{ij}\vert_{_{S(G_i)}}, \mathcal{I}\}$ is the associated surjective inverse system.
\end{lem}

\begin{proof} See Proposition $1$ in \cite{Por19}.
\end{proof}

The following facts are direct consequences of the definition of subgroups $S(G), S(L),\ S(G / J)$  and the Correspondence Theorem. In the case of abstract groups, these can be found in remark on  page 312 in \cite{Cap71} or Lemma 2.2 in \cite{kappe}. 

\begin{rem} \label{quocientecappitt}
If $G$ is an abstract group with $S(G) \neq G$ and $J$ is a normal subgroup of $G$ contained in $S(G)$, then $S(G/J) \neq G/J$, furthermore if $L$ is a subgroup of $G$ containing elements of $G \setminus S(G)$ then $S(L) \neq L$.  
\end{rem}

Analogously to the abstract case, similar results can be shown in the category of profinite groups, as follows. 

\begin{rem}\label{quocientecappittprofinite}
If $G$ is a profinite  Cappitt group (i.e. $\tilde{S}(G) \neq G$) and $J$ is a closed normal subgroup of $G$ contained in $\tilde{S}(G)$, then $\tilde{S}(G/J) \neq G/J$, furthermore if $L$ is a closed subgroup of $G$ containing elements of $G \setminus \tilde{S}(G)$ then $\tilde{S}(L) \neq L$. \end{rem}

The following result is similar to Lemma $1$ in \cite{Por19}, we will repeat the same argument for the convenience of the reader.

\begin{lem} Let $G$ be as in the Lemma \ref{limitesdeS}. If $G \neq \tilde{S}(G)$ then $G_i \neq S(G_i),\,\, \forall i \in \mathcal{I}$.

\label{commutador2}
\end{lem}

\begin{proof} Suppose that there is a $G_i$ such that $\tilde{S}(G_i)=S(G_i) = G_i$. Denote by $\phi_{i}: G \twoheadrightarrow G_i$ the continuous canonical projection of the inverse limit described in Lemma \ref{limitesdeS}. By assumption $G_i$ is generated by all its non-normal subgroups. Let $\bar{L}$ be one of these subgroups and consider $L_{i}$ its inverse image by $\phi_{i}$, clearly $L_{i}$ is a non-normal closed subgroup of $G$ containing $\ker(\phi_{i})$. Therefore $\ker(\phi_{i}) \leqslant \tilde{S}(G)$ and by Remark \ref{quocientecappittprofinite} have $G/\ker(\phi_i) \neq \tilde{S}(G/\ker(\phi_i))=S(G/\ker(\phi_i)),$ so $|S(G_i)| < |G_i|<\infty$, a contradiction. 
\end{proof}

In accordance with the notations of the previous lemmas we obtain the following.

\begin{lem} Let $G$ be a pro-$p$ Cappitt group. Then $G\,'$ is a procyclic central subgroup.

\label{commutador}
\end{lem}

\begin{proof} Since Cappitt groups are nilpotent of class at most $2$ (see Corollary 1 in \cite{Por19}) we have $G'=\overline{[G,G]} \leqslant Z(G).$ Consider $G = \displaystyle\lim_{\longleftarrow\atop{i \in \mathcal{I}}}\,\{G_i, \phi_{ij}\}$ where $\{G_i, \phi_{ij}, \mathcal{I}\}$ is a surjective inverse system of finite $p$-groups (with discrete topology). From Lemma \ref{commutador2} it follows that $G_i \neq S(G_i),\,\, \forall i \in \mathcal{I}$. Now, if any $G_i$ is abelian we have $(G_i)\,' =\{ 1\}$. On the other side, if $G_i$ is a non-abelian $p$-group then from Theorem $1$ in \cite{Cap71}, we have $(G_i)'$ is a finite cyclic $p$-group. Note that $\phi_i(G\,') = (G_i)\,',  \forall i \in \mathcal{I}$, hence from Corollary $1.1.8$ in \cite{RZ10} we have $G\,' = \displaystyle\lim_{\longleftarrow\atop{i \in \mathcal{I}}}\,\{(G_i)', \phi_{ij}\}$ whence it follows that $G\,'$ is procyclic. 
\end{proof}

The preceding result is a generalization of the Corollary on page 314 in Cappitt \cite{Cap71}. Kappe and Reboli show that the proof given by  Cappitt was incorrect and give a new proof for this fact, this is the content of Theorem 4.2 in Kappe and Reboli \cite{kappe}.

\begin{rem} \label{problemmm}
If $G$ is a non-abelian pro-$p$ Cappitt group, then $G$ is nilpotent of class $2$ (see Corollary 1 in  \cite{Por19}).  In this case $G\,' \leqslant Z(G)$ and therefore $[[x,y],z]=1, \, \forall x, y, z \in G$. By Lemma 5.42 (i) in Rotman (see \cite{rotman}), we have that $[x^n,y]=[x,y^n]=[x,y]^n, \, \forall n \in \mathbb{Z}.$ 
\end{rem}

\begin{lem} \label{ecentral}
Let $G$ be a non-abelian pro-$p$ Cappitt group. If $a \in G \setminus \tilde{S}(G)$ has infinite order, then $a$ is central.
\end{lem}
\begin{proof}
Let $a$ be an element in $G \setminus \tilde{S}(G)$ of infinite order, hence $|a|=p^{\infty}$ and $\overline{\left\langle a \right\rangle} \cong \mathbb{Z}_{p}$. If $a$ is not central in $G$ then there is $g \in G$ such that $[a,g] \neq 1$, moreover $g^{-1}ag \in \overline{\left\langle a \right\rangle}$ because $\overline{\left\langle a \right\rangle} \triangleleft G$. By Remark \ref{problemmm} we have $g^{-1}aga^{-1}=[g,a^{-1}]={[g,a]}^{-1}=[a,g] \in \overline{\left\langle a \right\rangle}$. Since $\mathbb{Z}_p$ is a free pro-$p$ group on $\{1\}$, there is a unique continuous epimorphism $\zeta: \mathbb{Z}_p \longrightarrow \overline{\left\langle a \right\rangle}$ such that $\zeta(1)=a$, therefore each element of $\overline{\left\langle a \right\rangle}$ can be written as $a^{\lambda} = \zeta(\lambda)$ for some $\lambda \in \mathbb{Z}_p$ (see similar notation in Section 4.1 on \cite{RZ10}).  Since $\overline{\left\langle a \right\rangle}$ is infinite it follows that $\zeta$ is an isomorphism (see Proposition 2.7.1 in \cite{RZ10}). In particular $1 \neq [a,g]=a^{\delta}$ for some  $\delta \in \mathbb{Z}_p$. 
Let $(n_i)_{_{i \in \mathbb{N}}}$ be a sequence of integers converging to $\delta$ on $\mathbb{Z}_{p},$ in other words choose the $n_i$ such that $\displaystyle\lim_{i \longrightarrow \infty } \, n_i= \delta$ (note that $ \overline{\left\langle a \right\rangle} \cong \mathbb{Z}_{p}$ is a complete metric space and $\overline{\mathbb{Z}} = \mathbb{Z}_{p}$). Adapting the argument from Lemma 4.1.1 in \cite{RZ10} to $\mathbb{Z}_p$ we have $\displaystyle\lim_{i \longrightarrow \infty } \,[a,g]^{n_i}= [a,g]^{\delta}.$ By  properties of limits of sequences in metric spaces together with Remark \ref{problemmm} we obtain $$[a,g]^{\delta}=\displaystyle\lim_{i \longrightarrow \infty } \,[a^{n_i},g]=\displaystyle\lim_{i \longrightarrow \infty }(a^{n_i})^{-1} \cdot \displaystyle\lim_{i \longrightarrow \infty } g^{-1} \cdot \displaystyle\lim_{i \longrightarrow \infty } a^{n_i} \cdot  \displaystyle\lim_{i \longrightarrow \infty } g=[a^{\delta},g].$$ As $G'$  is central (see Remark \ref{problemmm}) we have $[a,g]=a^{\delta} \in Z(G).$ So  $[[a,g],g]=[a^{\delta},g]=[a,g]^{\delta}=1,$ which implies that $a^{2\cdot \delta }=\zeta(2 \cdot \delta)=1,$ a contradiction because $Ker(\zeta)=\{0\}$. Therefore $a$ is central in $G.$ 
\end{proof}

\begin{prop} Let $G$ be a non-abelian pro-$p$ Cappitt group and $tor(G)$ is a closed subgroup of $G$. Then $G$ is a periodic group. 

\label{torsionprop}
\end{prop}

\begin{proof} Suppose that $G$ contains elements of infinite order. Since $G$ is nilpotent whose class is two (see Corollary $1$ in \cite{Por19}), it is generated by these elements, not all of which lie in $\tilde{S}(G)$ otherwise we would have $G= \overline{\left\langle G \setminus \tilde{S}(G) \right\rangle} =\overline{\left\langle tor(G)\right\rangle}=tor(G)$, a contradiction. Let $a$ be an element in $G \setminus \tilde{S}(G)$ of infinite order. From Lemma \ref{ecentral} $a$ is central. Now if $w$ is a non-central element of $G \setminus  \tilde{S}(G)$ then $aw, a^2w$ are non-central elements of infinite order, so both lie in $\tilde{S}(G)$. Thus $a^2ww^{-1}a^{-1}=a \in\tilde{S}(G)$,  contradicting the choice of $a$. Therefore every element of $G \setminus \tilde{S}(G)$ is central and thus $G$ is abelian since $G \setminus \tilde{S}(G)$ generates $G$, a contradiction. 
\end{proof}

\begin{ex}
Note that the condition $\tilde{S}(G) \neq G$ in the Proposition \ref{torsionprop} is necessary, for instance $G= \left( \displaystyle\prod_{r \in \mathbb{N}}\,Z_r\right) \times Q_8$ is a non-abelian pro-$2$ non-periodic group with $tor(G) = Q_8$ and $Z_r \cong \mathbb{Z}_2, \forall r \in \mathbb{N}$. We show that $\tilde{S}(G)=G$. Fix $n \in \mathbb{N}$, indeed, it is sufficient we consider the following families of procyclic subgroups of $G$  isomorphic to $\mathbb{Z}_2$, that we will denote by  $$I_n=\overline{\langle(0,\ldots, 0, 1_n, 0,  \ldots,0, i)\rangle},$$ $$J_n=\overline{\langle(0,\ldots,0,1_n,0, \ldots,0,j)\rangle},$$  
$$K_n=\overline{\langle(0,\ldots,0,1_n,0, \ldots,0,k)\rangle},$$ where  $1_n:=(0,\ldots,0,1_n,0,\ldots,1)$ is the generator of $Z_n \cong \mathbb{Z}_2$. Observe that $$(0,\ldots,0,1_n,0,\ldots,0,i)^j = (0,\ldots,0,1_n,0,\ldots,0,-i) \not \in I_n,$$ so inductively $I_n$ is not a normal subgroup of $G$. Similarly, it can be shown also that $J_n$ and $K_n$ are not normal subgroups of $G$ for each $n \in \mathbb{N}$. Therefore by the definition of $\tilde{S}(G)$ we have that $I_n, \ J_n$ and $K_n$ are contained in $\tilde{S}(G)$. Now consider the following elements of $\tilde{S}(G)$, $x_n=(0,\ldots,0,1_n,0,\ldots, i)$ and $y_n=(0,\ldots,0,1_n,0,\ldots,0,j)$. Note that $x_n \cdot y_n^{-1} = (0,\ldots,0,-k):=-k$, so $k \in \tilde{S}(G)$. Similarly $i, j \in \tilde{S}(G)$ and this implies that $Q_8$ is contained in $\tilde{S}(G)$. Since $(0,\ldots,0,1_n,0,\ldots,0,k)$ and $(0, \ldots,0,-k)$ are in $\tilde{S}(G)$ we have $1_n:=(0, \ldots, 1_n,0,\ldots,0,1)$ belongs to $\tilde{S}(G)$, and therefore $\tilde{S}(G) = G$.

\label{exampleperiodic}
\end{ex}


\section{Proof of the Theorem \ref{teoA}}

We are now ready to prove a main result of this paper.\\
 
\begin{proof} The fact that $G'$ is a procyclic central subgroup follows from Lemma \ref{commutador}. If $G$ is Dedekind (i.e. pro-$2$ Hamiltonian group) the result follows immediately from Theorem \ref{dedekindBP} and Remark \ref{matalogo}. From now on we will consider that $\tilde{S}(G)\neq \{1\}$. Since $tor(G)$ is closed in $G$, from Proposition \ref{torsionprop} follows that $G$ is a pro-$p$ torsion group. If $G$ is finitely generated as pro-$p$ group then from celebrated Zel’manov Theorem (see Theorem 4.8.5c in \cite{RZ10}) we have that $G$ is a finite $p$-group and the result follows. Therefore we can suppose that $G$ is an infinitely generated pro-$p$ and $\{1\} \neq \Tilde{S}(G) \neq G$. Also since $G$ is a torsion group and $G'$ is procyclic, it follows from Proposition 2.7.1 \cite{RZ10} that $G\,' = C_{p^n}$ for some $n \in \mathbb{N}$ fixed.  

Suppose by contradiction that $G$ has no finite exponent. So there are $p$-elements of $G$, say, $x_1, x_ 2, \ldots, x_k, \ldots, (k \in \mathbb{N})$, such that $|x_k| \rightarrow \infty$ for $k \rightarrow \infty$ (elements of unlimited order). 

Let $G^{ab}=G/\overline{[G,G]}$ be the abelianization of $G$. It is straightforward prove that $G^{ab}$ is an abelian pro-$p$ torsion group because $G$ is a torsion group.  From Theorem $4.3.8$ in \cite{RZ10} we obtain that $$G^{ab} = \prod_{m(1)} C_p \times \prod_{m(2)} C_{p^2} \times \ldots \times \prod_{m(e)} C_{p^e},$$ where each $m(i)$ is a cardinal number, $e$ is some natural number, with at least one $m(i)$ infinite because $G^{ab}$ is infinitely generated.

For each $g \in G$ we have    $(gG\,')^{p^e} = G\,'$ since $p^e$ is the finite exponent of $G^{ab}$. Therefore for all $k \in \mathbb{N}$ we have $$G\,' = (x_kG\,')^{p^e}=x_k^{p^e}G' \Longrightarrow x_k^{p^e} \in G\,'=C_{p^n}.$$ Thus we obtain  $$({x_k}^{p^e})^{p^n}= x_k^{p^{e+n}}=1.$$ Therefore the order of each $x_k$ is limited by $p^{e+n}, \forall k \in \mathbb{N},$ a contradiction. We conclude that $G$ has finite exponent. \end{proof}

\section{Pro-$2$ Cappitt groups of exponent $4$}

The following result shows that the only pro-$2$ Cappitt groups of exponent $4$ are pro-$2$ Dedekind groups. 

\begin{lettertheorem}\label{teoB} If $G$ is a pro-$2$ group of exponent $4$ then either $\tilde{S}(G)=\{1\}$ or $\tilde{S}(G)=G$.
\end{lettertheorem}

\begin{proof} Suppose that $\{1\} \neq \tilde{S}(G) \neq G$. It is clear that not every procyclic subgroup of $G$ is normal, otherwise $\tilde{S}(G) = \{1\}$ (see Remark \ref{remark prociclica}). Then consider $J = \overline{\langle x \rangle}$ a non-normal procyclic subgroup of $G$. Let $H= \overline{\langle x,y,z\rangle} \leqslant_{c} G$ such that $y \not\in N_{G}(J)$ and $z \not\in \tilde{S}(G)$. By construction and Remark \ref{quocientecappitt} we have that $H$ is a $3$-generated group of exponent $4$ with $\{1\} \neq \tilde{S}(H) \neq H$ and $\overline{\langle z\rangle} \trianglelefteq_{c} H$. Since $H$ has exponent $4$, we have that $H$ is a finitely generated pro-$2$ torsion group, it follows from celebrated Zel’manov Theorem that $H$ is a finite non-abelian $2$-group (see Theorem 4.8.5c in \cite{RZ10}). From Corollary 1 in Porto \cite{Por19} and Lemma \ref{commutador} we have that $G$ and $H$ are nilpotent groups of class $2$ with $H'$ being cyclic, so $H' \cong C_2$ or $H' \cong C_4$. Take a minimal set of generators of $H$ chosen to lie outside $\tilde{S}(H)=S(H)$, say, $\{h_1, h_2, \ldots, h_l\}$, where $l \geq 2$. Note that $H'=\left \langle [h_i, h_j]\,\mid \,i,j \in\{1,2, \ldots,l\}, i \neq j\right \rangle$ (see p. 129 in \cite{robinson}). Let $H_i=\left \langle h_i \right \rangle$  for each $i \in \{1,2, \ldots,l\}$.  Since $h_i \not \in S(H)$ we have $H_i \trianglelefteq H$ and so $[h_i, h_j] \in H_i \cap H_j$ for each $i,j \in\{1,2, \ldots,l\}$. 
Note that $[h_i, h_j]$ belongs to $H_m$ for every $i, j, m\in \{1,2,\ldots,l\}$, so $H'\leqslant H_m$ for each $m\in \{1,2,\ldots,l\}$. If $H' \cong C_4$ we have that $H'=H_1=H_2=\cdots=H_m$, a contradiction since $H$ is not abelian. Therefore $H'$ has order $2$. Note that $H$ is a finite $2$-group and it is $3$-generated with exponent $4$, so it is well-know that its order is at most $64$. A GAP computation using \cite{Gap} shows that $H$ is either Dedekind or $\tilde{S}(H)=S(H) = H$, a contradiction. 
\end{proof}

For example, a GAP \cite{Gap} check yields, the SmallGroup(16, 6), SmallGroup(27, 4) and SmallGroup(64, 28) are finite  Cappitt groups with commutator subgroup $C_2$, $C_3$ and $C_4$, respectively.

In Corollary $4.4$ in \cite{kappe}, Kappe and Reboli prove that if $G$ is a $p$-group Cappitt, then $G'$ is a finite cyclic $p$-group. Finally, in view of this result, example above,  Lemma \ref{commutador} and Theorem \ref{teoA}, we can pose the following questions.

\section{Open questions about profinite Cappitt groups}

\begin{enumerate}
    \item Are there non-abelian non-periodic pro-$p$ Cappitt groups? If $G$ is not torsion, is $G$ virtually $p$-adic?
   
    \item If $G$ is a pro-$p$ group satisfying $\{1\}\neq \tilde{S}(G) \neq G$, do we always have that $G'  \cong C_{p^n}$ for some $n \in \mathbb{N}$?
    \end{enumerate}


\subsection*{Acknowledgment}
We thank John MacQuarrie for the careful reading and suggestions.


\begin{thebibliography}{99}

\bibitem{Cap71} D. Cappitt. Generalized Dedekind groups. \textit{Journal of Algebra}, \textbf{17} (1971), 310–316.




\bibitem{Herf80} W. N. Herfort. Compact torsion groups and finite exponent. \textit{Archiv der Mathematik}, \textbf{33} (1980), 404–410.

\bibitem{Herf82} W. N. Herfort. An arithmetic property of profinite groups. \textit{Manuscr. Math.}, \textbf{37} (1982), 11–17.

\bibitem{kappe} L. Kappe and D. M. Reboli, On the structure of generalized Hamiltonian groups. \textit{Archiv der  Mathematik}, \textbf{75} (2000), 328-337.


\bibitem{Kargapolov} M. I. Kargapolov and Ju. I. Merzljakov. Fundamental of the Theory of Groups. Springer-Verlag. New York. Heidelberg. Berlin, 1979. 



\bibitem{Por19} A. L. P. Porto. Profinite Cappitt groups. \textit{Quaestiones Mathematicae}, v. 44 (\textbf{3}) (2021), 307-311.

\bibitem{PB19} A. L. P. Porto and V. R. de Bessa. Profinite Dedekind groups. \textit{Far East Journal of Mathematical Sciences (FJMS)}, v. 126 \textbf{(1)} (2020), 89-97. 

\bibitem{RZ10} L. Ribes and P.A. Zalesskii. Profinite groups. Second edition, \textit{Springer-Verlag}, Berlin Heidelberg, 2010.

\bibitem{robinson} D. J. S. Robinson. A course in the Theory of Groups. 2nd. ed. Springer-Verlag. New York. Heidelberg. Berlin, 1996. 

\bibitem{rotman} J. J. Rotman. An introduction to the theory of Groups. 4 th. ed. Springer-Verlag. New York. Heidelberg. Berlin, 1994. 


\bibitem{Gap} The GAP Group (2019). GAP – Groups, Algorithms, and Programming, Version 4.10.2. https://www.gap-system.org

\bibitem{wilson98} J. S. Wilson. Profinite Groups, \textit{Claredon Press}, Oxford, 1998.

\bibitem{Zelm92} E. I. Zel’manov. On periodic compact groups. \textit{Isr. J. Math.}. \textbf{77} (1992), 83–95.







\end{thebibliography}
\end{document}